\documentclass[OJMO,manuscript]{cedram}

\usepackage[all]{xypic}

\usepackage[capitalize, nameinlink, noabbrev]{cleveref}
\crefname{equation}{}{}
\crefname{chapter}{Chapter}{Chapters}
\crefname{item}{item}{items}
\crefname{figure}{Figure}{Figures}
\crefname{theo}{Theorem}{Theorems}
\crefname{lemm}{Lemma}{Lemmas}
\crefname{prop}{Proposition}{Propositions}
\crefname{corollary}{Corollary}{Corollaries}
\crefname{definition}{Definition}{Definitions}
\crefname{fact}{Fact}{Facts}
\crefname{exam}{Example}{Examples}
\crefname{algorithm}{Algorithm}{Algorithms}
\crefname{rema}{Remark}{Remarks}
\crefname{note}{Note}{Notes}
\crefname{notation}{Notation}{Notations}
\crefname{case}{Case}{Cases}
\crefname{exercise}{Exercise}{Exercises}
\crefname{question}{Question}{Questions}
\crefname{claim}{Claim}{Claims}
\crefname{enumi}{}{}

\title[Characterizations of Stability of Error Bounds]{Characterizations of Stability of Error Bounds for Convex Inequality Constraint Systems}

\author{\firstname{Zhou} \middlename{} \lastname{Wei}}
\address{College of Mathematics and Information Science, Hebei University, Baoding 071002, China}
\email{weizhou@hbu.edu.cn}
\address{Department of Mathematics, Yunnan University,  Kunming 650091, China}
\email{wzhou@ynu.edu.cn}
\thanks{The first author was supported by the National Natural Science Foundation of  China (grant 11971422) and CAS ``Light of West China" Program, and by Jointed Key Project of Yunnan Provincial Science and Technology Department and Yunnan University [No. 2018FY001014] and Program for Innovative Research Team (in Science and Technology) in Universities of Yunnan Province [No. C176240111009].}
\author{\firstname{Michel} \lastname{Th\'era}}
\address{XLIM UMR-CNRS 7252, Universit\'e de Limoges\\Limoges, France \\Federation University Australia, Ballarat}
\thanks{Research of the second author benefited from the support of the FMJH Program PGMO and from the
	support of EDF}
\email{michel.thera@unilim.fr}
\author{\firstname{Jen-Chih} \middlename{} \lastname{Yao}}
\address{Research Center for Interneural Computing, China Medical University Hospital\\
	China Medical University, Taichung 40402, Taiwan}
\email{yaojc@mail.cmu.edu.tw}

  
\begin{abstract} 
 In this paper, we mainly study error bounds for a single convex inequality and semi-infinite convex constraint systems, and give characterizations of stability of error bounds via directional derivatives. For a single convex inequality, it is proved that the stability of local error bounds under small perturbations is essentially equivalent to the non-zero minimum of the directional derivative at a reference point over the unit sphere, and the stability of global error bounds is proved to be equivalent to the strictly positive infimum of the directional derivatives, at all points in the boundary of the solution set, over the unit sphere as well as some mild constraint qualification. When these results are applied to semi-infinite convex constraint systems, characterizations of stability of local and global error bounds under small perturbations are also provided. In particular such stability of error bounds is proved to only require that all component functions in  semi-infinite convex constraint systems have the same linear perturbation. Our work demonstrates that verifying the stability of error bounds for convex inequality constraint systems is, to some degree, equivalent to solving convex minimization problems (defined by directional derivatives) over the unit sphere.

\end{abstract}

\begin{document}

\maketitle

\section{Introduction}

Our main goal  in this paper is to study error bounds of a single convex inequality and semi-infinite convex constraint systems and  to provide characterizations of stability of local and global error bounds under perturbations. Theory of error bounds can be traced back to the pioneering work by Hoffman \cite{28} for systems of affine functions in which it has been proved that for a given matrix $A$ and a vector $b$, the distance from $x$ to the polyhedral set $\{u:Au\leq b\}$ is bounded above by some scalar constant (depending on $A$ only) times the norm of the residual error $\|(Ax-b)_+\|$, where for any vector $z$, $(z)_+$ denotes the positive part of $z$.

Hoffman's result was extensively and intensively studied by Robinson \cite{54}, Mangasarian
\cite{42}, Auslender and Crouzeix \cite{1}, Pang \cite{52}, Lewis and Pang \cite{41}, Klatte and Li \cite{34},
Jourani \cite{jourani}, and there have been important developments of various aspects of error bounds for convex and nonconvex functions in recent years. We refer the readers to bibliographies \cite{Aze01,Aze02,3,7,18,CK,DL,22,24,29,ioffe-book, Kruger-LY, Luke,44,48,penot-book,55} and references therein for the summary of the theory of error bounds and their various applications for more details.

Error bounds have been applied  to the sensitivity analysis of linear programs (cf. \cite{Rob73,Rob77}) and to the  convergence analysis of descent methods for linearly constrained minimization (cf. \cite{Gul92,HLu,IuD90,TsL92,TsB93}).  In addition, it is proved that error bounds   play an important role in the feasibility problem of finding a point in the intersection of a finite collection of closed convex sets (cf. \cite{5,6,7}) and have an application in the domain of image reconstruction (cf. \cite{16}). Also,  error bounds are extensively discussed in connection with weak sharp minima of functions and metric regularity/subregularity as well as Aubin property/calmness of set-valued mappings (cf. \cite{AT,3,BD1,BD2,14,24,ioffe-JAMS-1,ioffe-JAMS-2,K1,K2,54,59,60} and references therein).

Since real-world problems typically have inaccurate data, it is of practical and theoretical interest to know the behavior of error bounds under data perturbations. For systems of linear inequalities, this question has been studied by Luo and Tseng \cite{LT} and Az\'e and Corvellec \cite{2}. Subsequently Deng \cite{D} studied systems of a finite number of convex inequalities. In 2005, Zheng and Ng \cite{ZN} considered the stability of error bounds for systems of conic linear inequalities in a general Banach space. In 2010, Ngai, Kruger and Th\'era \cite{45} studied the stability of error bounds for semi-infinite convex constraint systems in a Euclidean space and established subdifferential characterizations of the stability under small perturbations. The infinite dimensional extensions were considered by Kruger, Ngai and Th\'era in \cite{38}. In 2012, by relaxing the convexity assumption, Zheng and Wei \cite{ZW} discussed the stability of error bounds for quasi-subsmooth (not necessarily convex) inequalities in a general Banach space and provided Clarke subdifferential characterizations of the stability of error bounds. In 2018, Kruger, L\'opez and Th\'era \cite{MP2018} extended the development in \cite{38,45} and characterized the stability of error bounds for convex inequalities in the Banach space setting. From the viewpoint of infinite dimensional Banach spaces, results on the stability of error bounds in \cite{38,MP2018,45,ZW} are dual conditions, and it is a pretty natural idea to study this issue not involving the dual space since information on the dual space may be missing. Inspired by this observation, we study characterizations of stability of local and global error bounds of a single convex inequality and semi-infinite convex constraint systems via directional derivatives. For a single convex inequality, we prove that the stability of local error bounds under small perturbations holds if and only if the minimum of the directional derivative at a reference point over the unit sphere is non-zero, and the stability of global error bounds is proved to be equivalent to the strictly positive infimum of the directional derivatives, at all points in the boundary of the solution set, over the unit sphere as well as some mild constraint qualification. When these results are applied to semi-infinite convex constraint systems, characterizations of the stability of local and global error bounds under small perturbations are also provided. Particularly such stability of error bounds is proved to only require that all component functions in  semi-infinite convex constraint systems have the same linear perturbation. Our work demonstrates that verifying the stability of error bounds for convex inequality constraint systems is, to some degree, equivalent to solving optimization/minimization problems (defined by directional derivatives) over the unit sphere.

The paper is organized as follows. In Section 2, we give some definitions and preliminary results. Section 3 is devoted to the study on stability of error bounds for a single convex inequality. In terms of directional derivatives, we provide characterizations of local and global error bounds for a single convex inequality under small perturbations(see \cref{theo:3.2} and \cref{theo:3.4}). When these results are applied to the semi-infinite convex constraint systems in Section 4, the stability of local and global error bounds can be obtained (see  \cref{theo:4.1} and  \cref{theo:4.2}). Conclusions of this paper are given in Section 5.


\setcounter{equation}{0}
\section{Preliminaries}
In what follows we consider the Euclidean space $\mathbb{R}^m$  equipped with the norm $\|\cdot\|:=\sqrt{\langle\cdot, \cdot\rangle}$. We denote by  $\mathbf{B}^m$ the closed unit ball of $\mathbb{R}^m$  and  following the standard notation by $\Gamma_0(\mathbb{R}^m)$ the set of extended-real-valued lower semicontinuous convex functions $f:\mathbb{R}^m\rightarrow \mathbb{R}\cup\{+\infty\}$ which are supposed to be proper, that is such that ${\rm dom}(f):=\{x\in\mathbb{R}^m: f(x)<+\infty\}$ is nonempty.

For a subset $D$ of $\mathbb{R}^m$, we denote by $d(x,D)$ the distance from $x$ to $D$ which is defined by
$$
d(x,D):=\inf\{\|x-y\|:y\in D\},
$$
where we use the convention $\inf\emptyset=+\infty$. We denote by ${\rm bdry}(D)$  and ${\rm int}(D)$ the  boundary and the interior of $D$, respectively.

Let $f\in\Gamma_0(\mathbb{R}^m)$ and $\bar x\in{\rm dom}(f)$. For any $h\in \mathbb{R}^m$, we recall that the  directional derivative  $f'(\bar x,h)$  of $f$ at $\bar x$ along  the direction $h$  is defined as
\begin{equation}\label{2.1}
	f'(\bar x,h):=\lim\limits_{t\rightarrow0^+}\frac{f(\bar x+th)-f(\bar x)}{t}.
\end{equation}
It is known from \cite{roc} that the function
$$
t\mapsto \frac{f(\bar x+th)-f(\bar x)}{t}
$$
is nonincreasing as $t\rightarrow 0^+$ and thus
\begin{equation}\label{2.2}
	f'(\bar x,h)=\inf_{t>0}\frac{f(\bar x+th)-f(\bar x)}{t}.
\end{equation}
We denote by $\partial f(\bar x)$ the subdifferential of $f$ at $\bar x$  which is defined by
$$
\partial f(\bar x):=\{x^*\in \mathbb{R}^m: \langle x^*,x-\bar x\rangle\leq f(x)-f(\bar x)\ \ {\rm for\ all} \ x\in \mathbb{R}^m\}.
$$
It is known from \cite{Moreau,roc} that 
\begin{equation}\label{2.3}
	\partial f(\bar x)=\{x^*\in\mathbb{R}^m: \langle x^*,h\rangle\leq f'(\bar x,h)\ \ {\rm for\ all} \ h\in \mathbb{R}^m\}
\end{equation}
and if $f$ is continuous at $(\bar x)$, one has
\begin{equation}\label{2.4}
f'(\bar x,h)=\max\{ \langle x^*,h\rangle: x^* \in\partial f(\bar x)\}.
\end{equation}

We conclude this section with the following lemma which is used in our analysis.
\begin{lemm}\label{lemm:2.1}
Let $f\in\Gamma_0(\mathbb{R}^m)$ and $\bar x\in {\rm dom}(f)$ be such that $\inf_{\|h\|=1}f'(\bar x,h)<0$. Then
\begin{equation}\label{2.5}
	-\inf_{\|h\|=1}f'(\bar x,h)=d(0,\partial f(\bar x)).
\end{equation}
\end{lemm}

\begin{proof}
We denote $\alpha:=\inf_{\|h\|=1}f'(\bar x,h)$.

If $\alpha=-\infty$, then one has $\partial f(\bar x)=\emptyset$ by \eqref{2.3} and thus \eqref{2.5} holds.

Next, we consider the case $\alpha>-\infty$.  Note that $\alpha<0$ and thus $0\not\in\partial f(\bar x)$. Let $r:=d(0,\partial f(\bar x))$. For any  $\varepsilon>0$, we can select $u^*_{\varepsilon}\in (r+\varepsilon)\mathbf{B}^m\cap\partial f(\bar x)$. Then for any $h\in\mathbb{R}^m$ with $\|h\|=1$, one has
$$
f'(\bar x,h)\geq \langle u^*_{\varepsilon}, h\rangle\geq -(r+\varepsilon)
$$
and consequently
$$
\inf_{\|h\|=1}f'(\bar x,h)\geq -(r+\varepsilon).
$$
By letting $\varepsilon\rightarrow 0^+$, it follows that $d(0,\partial f(\bar x))\geq-\alpha$. 

We now assume that $d(0,\partial f(\bar x))>-\alpha>0$. Then $\bar x$ is not a global minimizer of $f$. We claim that there exists $y\in \mathbb{R}^m$ such that
\begin{equation}\label{1}
	f(y)-f(\bar x)<\alpha\|y-\bar x\|.
\end{equation}
(Indeed, suppose on the contrary that
$$
f(x)-f(\bar x)-\alpha\|x-\bar x\|\geq 0,\forall x\in \mathbb{R}^m.
$$
This implies that 
$$
\varphi(\bar x)=\min_{x\in \mathbb{R}^m}\varphi(x),
$$
where $\varphi(x):=f(x)-f(\bar x)-\alpha\|x-\bar x\|$ and thus $0\in\partial\varphi(\bar x)$. Then, there exist $x^*\in\partial f(\bar x)$ and $u^*\in \mathbf{B}^m$ such that
$$
x^*-\alpha u^*=0,
$$
which means that $d(0,\partial f(\bar x))\leq\|x^*\|\leq-\alpha$, a contradiction with $d(0,\partial f(\bar x))>-\alpha$.)

Using the convexity of $f$, when $t>0$ is sufficiently small, one has
$$
\frac{f(\bar x+t\frac{y-\bar x}{\|y-\bar x\|})-f(\bar x)}{t}\leq\frac{f(y)-f(\bar x)}{\|y-\bar x\|}<\alpha
$$
and thus
$$
f'(\bar x,\frac{y-\bar x}{\|y-\bar x\|})<\alpha
$$
which is a contradiction. This means that \eqref{2.5} holds. The proof is complete.
\end{proof}

\section{Stability of Error Bounds for A Single Convex Inequality}
In this section, we mainly study local and global error bounds for a single convex inequality, and provide characterizations of stability (in terms of directional derivatives) of error bounds. We first recall the definition of error bounds for a single convex inequality.

For a given $f\in\Gamma_0(\mathbb{R}^m)$, we consider the set of solutions of a single convex inequality:
\begin{equation}\label{3.1}
	S_f:=\{x\in\mathbb{R}^m: f(x)\leq 0\}.
\end{equation}
Recall that convex inequality \eqref{3.1} is said to have a \textit{global error bound} if there exists a constant $\tau\in(0,+\infty)$ such that
\begin{equation}\label{3.2}
	d(x, S_f)\leq\tau [f(x)]_+\ \ \forall x\in\mathbb{R}^m,
\end{equation}
where $[f(x)]_+:=\max\{f(x), 0\}$. We denote by $\tau_{\min}(f):=\inf\{\tau>0: \eqref{3.2}\ {\rm holds}\}$ the global error bound modulus of $S_f$.

For $\bar x\in {\rm bdry}(S_f)$, convex inequality \eqref{3.1} is said to have a \textit{local error bound} at  $\bar x$ if there exist  constants $\tau,\delta\in(0,+\infty)$ such that
\begin{equation}\label{3.3}
	d(x, S_f)\leq\tau [f(x)]_+\ \ \forall x\in B(\bar x,\delta).
\end{equation}
We denote by $\tau_{\min}(f,\bar x):=\inf\{\tau>0:\ {\rm there\ exists} \ \delta>0\ {\rm such\ that} \ \eqref{3.3}\ {\rm holds}\}$ the local error bound modulus of $S_f$ at $\bar x$.


The following theorem gives characterizations of global and local error bounds. We refer the readers to \cite{2} for more details. This result is needed in the sequel.

\begin{theo}\label{theo:3.1}
	Let $f\in\Gamma_0(\mathbb{R}^m)$. Then
	\begin{itemize}
		\item[\rm (i)] $S_f$ has a global error bound if and only if
		$$
		\eta(f):=\inf\{d(0,\partial f(x)): x\in\mathbb{R}^m, f(x)>0\}>0.
		$$
		More precisely, $\tau_{\min}(f)=[\eta(f)]^{-1}$.
		\item[\rm (ii)] $S_f$ has a local error bound at $\bar x\in{\rm bdry}(S_f)$ if and only if
		$$
		\eta(f,\bar x):=\liminf_{x\rightarrow\bar x, f(x)>0}d(0,\partial f(x))>0.
		$$
		More precisely, $\tau_{\min}(f,\bar x)=[\eta(f,\bar x)]^{-1}$.
		\item[\rm (iii)] The following equality holds:
		$$
		\tau_{\min}(f)=\sup_{\bar x\in{\rm bdry}S_f}\tau_{\min}(f,\bar x)
		$$
	\end{itemize}
\end{theo}

For a mapping $\phi:X\rightarrow Y$ between two normed linear spaces $X,Y$, we denote by ${\rm Lip}(\phi)$ the Lipschitz constant which is defined by
$$
{\rm Lip}(\phi):=\sup_{u,v\in X,u\not=v}\frac{\|\phi(u)-\phi(v)\|}{\|u-v\|}.
$$

\subsection{Stability of Local Error bounds}

In this subsection, we mainly study local error bounds for a single convex inequality and aim to provide equivalent criterion for the stability of local error bounds for convex inequality \eqref{3.1}. We first give a sufficient condition for the local error bound of convex inequality \eqref{3.1}.

\begin{prop}\label{prop:3.1}
	Let $f\in\Gamma_0(\mathbb{R}^m)$ and $\bar x\in S_f$ such that $\inf_{\|h\|=1}f'(\bar x, h)\not=0$. Then convex inequality \eqref{3.1} has a local error bound at $\bar x$ and moreover
	\begin{equation} \label{3.8}
		\tau_{\min}(f,\bar x)\leq \frac{1}{\Big|\inf\limits_{\|h\|=1}f'(\bar x, h)\Big|}.
	\end{equation}
\end{prop}

\begin{proof} Let $\beta(f,\bar x):=\inf_{\|h\|=1}f'(\bar x, h)$. Suppose that $\beta(f,\bar x)>0$. Then for any $x\not=\bar x$, by \eqref{2.2}, one can verify that
	\begin{eqnarray*}
		f(x)-f(\bar x)&=&f\Big(\bar x+\|x-\bar x\|\frac{x-\bar x}{\|x-\bar x\|}\Big)-f(\bar x)\\
		&\geq& f'\Big(\bar x,\frac{x-\bar x}{\|x-\bar x\|}\Big)\|x-\bar x\|\\
		&\geq&\beta(f,\bar x)\|x-\bar x\|\geq \beta(f,\bar x) d(x,S_f).
	\end{eqnarray*}
	This means that $\tau_{\min}(f,\bar x)\leq [\beta(f,\bar x)]^{-1}$.
	
	Suppose that $\beta(f,\bar x)<0$. Then  \text{\color{blue}Lemma 1} implies that $d(0,\partial f(\bar x))=-\beta(f,\bar x)$  and by virtue of  \cref{theo:3.1}, one has
	$$
	\tau(f,\bar x)\leq\frac{1}{-\beta(f,\bar x)}.
	$$
	Hence \eqref{3.8} holds. The proof is complete. \end{proof}


\begin{rema} Close analysis of the proof of \text{\color{blue}Proposition 3} shows that the solution set $S_f$ will reduce to the singleton $\{\bar x\}$ if $\inf_{\|h\|=1}f'(\bar x, h)>0$ and $f(\bar x)=0$, which means that $\bar x$ is the sharp (or strong) minimizer of $f$. Further, it should be noted that  the condition $\inf_{\|h\|=1}f'(\bar x, h)\not=0$ is only sufficient for the existence of  a local error bound of  \eqref{3.1}. Indeed, let $f(x)\equiv 0$ for all $x\in \mathbb{R}$. Then $S_f=\mathbb{R}$ has a global error bound, while $\inf_{\|h\|=1}f'(\bar x, h)=0$ for all $\bar x\in \mathbb{R}$. \hfill$\Box$\\
\end{rema}%

The following  theorem shows that the condition $\inf_{\|h\|=1}f'(\bar x, h)\not=0$ can be used to give characterizations of stability of the local error bound for  the convex inequality \eqref{3.1}. For the sake of completeness, we provide a self-contained proof of this theorem.
\begin{theo}\label{theo:3.2}
	Let $f\in\Gamma_0(\mathbb{R}^m)$ and $\bar x\in \mathbb{R}^m$ be such that $f(\bar x)=0$. Then the following statements are equivalent:
	\begin{itemize}
		\item[\rm (i)] $\inf_{\|h\|=1}f'(\bar x, h)\not=0$;
		\item[\rm (ii)] There exist constants $c,\varepsilon>0$ such that for all $g\in\Gamma_0(\mathbb{R}^m)$ satisfying $\bar x\in S_g$ and
		\begin{equation}\label{3.9}
			\limsup_{x\rightarrow \bar x}\frac{|(f(x)-g(x))-(f(\bar x)-g(\bar x))|}{\|x-\bar x\|}\leq \varepsilon,
		\end{equation}
		one has $\tau_{\min}(g,\bar x)\leq c$;
		\item[\rm (iii)] There exist constants $c,\varepsilon>0$ such that for all $u^*\in\mathbb{R}^m$ with $\|u^*\|\leq 1$, one has $\tau_{\min}(g_{u^*, \varepsilon},\bar x)\leq c$, where $g_{u^*, \varepsilon}(x):=f(x)+\varepsilon\langle u^*, x-\bar x\rangle$ for all $x\in \mathbb{R}^m$.
	\end{itemize}
\end{theo}

\begin{proof} Let $\beta(f,\bar x):=\inf_{\|h\|=1}f'(\bar x, h)$.

	(i)\,$\Rightarrow$\,(ii): Take any $\varepsilon>0$ such that $\varepsilon< |\beta(f,\bar x)|$ and let $c:=(|\beta(f,\bar x)|-\varepsilon)^{-1}$. For any $g\in\Gamma_0(\mathbb{R}^m)$ such that  $\bar x\in S_g$ and \eqref{3.9} holds. If $\beta(f,\bar x)>0$, then for any $h\in\mathbb{R}^m$, one has
	$$
	g'(\bar x, h)\geq f'(\bar x, h)-\varepsilon,
	$$
	and thus
	$$
	\inf_{\|h\|=1}g'(\bar x, h)\geq\inf_{\|h\|=1}f'(\bar x, h)-\varepsilon\geq \beta(f,\bar x)-\varepsilon.
	$$
	This and \text{\color{blue}Proposition 3} imply that $\tau_{\min}(g,\bar x)\leq [\beta(f,\bar x)-\varepsilon]^{-1}=c$.
	
	If $\beta(f,\bar x)<0$, then for any $h\in\mathbb{R}^m$, one has
	$$
	g'(\bar x, h)\leq f'(\bar x, h)+\varepsilon,
	$$
	and thus
	$$
	\inf_{\|h\|=1}g'(\bar x, h)\leq\inf_{\|h\|=1}f'(\bar x, h)+\varepsilon\leq \beta(f,\bar x)+\varepsilon.
	$$
	By using \text{\color{blue}Proposition 3} again, one yields that $\tau_{\min}(g,\bar x)\leq [-\beta(f,\bar x)-\varepsilon]^{-1}=c$. Hence (ii) holds.

	Note that the implication (ii)\,$\Rightarrow$\,(iii) is clear and it remains to prove (iii)\,$\Rightarrow$\,(i).
	
	Let $\varepsilon>0$. Suppose on the contrary that there exists a sequence $\{h_k\}$ in $\mathbb{R}^m$ with $\|h_k\|=1$ such that
	$$
	\alpha_k:=f'(\bar x,h_k)\rightarrow 0.
	$$
	Without loss of generality, we can assume that $|\alpha_k|<\varepsilon$ for all $k$ (considering sufficiently large $k$ if necessary) and consider the function $g_{\varepsilon}(x):=f(x)+\varepsilon \langle h_k, x-\bar x\rangle$ for all $x\in\mathbb{R}^m$. From $\beta(f,\bar x)=0$, one can verify that $f(x)\geq f(\bar x)$ for any $x\not=\bar x$. By the definition of directional derivative, there exists a sequence $\{\delta_k\}$ decreasing to $0$ such that
	\begin{equation}\label{3.10}
		f(\bar x+\delta_kh_k)<f(\bar x)+(\varepsilon+\alpha_k)\delta_k=\inf_{x\in\mathbb{R}^m}f(x)+(\varepsilon+\alpha_k)\delta_k.
	\end{equation}
	By virtue of the Ekeland variational principle, we can select $z_k\in\mathbb{R}^m$ such that  $ \|z_k-(\bar x+\delta_kh_k)\|<\frac{\delta_k}{2}, f(z_k)\leq f(\bar x+\delta_kh_k)$ and
	\begin{equation}\label{3.11}
		f(x)+2(\varepsilon+\alpha_k)\|x-z_k\|>f(z_k),\ \forall x\not=z_k.
	\end{equation}
	This implies that $z_k\rightarrow \bar x$, $g_{\varepsilon}(\bar x)=f(\bar x)=0$ and
	\begin{eqnarray*}
		g_{\varepsilon}(z_k)&=&f(z_k)+\varepsilon\langle h_k, z_k-\bar x \rangle\\
		&\geq& f(\bar x)+\varepsilon\langle h_k, z_k-\bar x \rangle\\
		&=&\varepsilon\langle h_k, z_k-\bar x-\delta_kh_k \rangle+\varepsilon\delta_k\\
		&>&\varepsilon\delta_k-\frac{1}{2}\varepsilon\delta_k=\frac{1}{2}\varepsilon\delta_k>0.
	\end{eqnarray*}
	We claim that
	\begin{equation}\label{3.12}
		\inf_{\|h\|=1}g_{\varepsilon}'(z_k,h)< 0.
	\end{equation}
	(Otherwise, $ \inf_{\|h\|=1}g_{\varepsilon}'(z_k,h)\geq 0$ and then one has $g_{\varepsilon}(z_k)=\inf_{x\in\mathbb{R}^m}g_{\varepsilon}(x)$, which contradicts $g_{\varepsilon}(\bar x)=0$).

	For any $h\in\mathbb{R}^m$ with $\|h\|=1$ and any $t>0$, by \eqref{3.11}, one has
	\begin{eqnarray*}
		\frac{g_{\varepsilon}(z_k+th)-g_{\varepsilon}(z_k)}{t}=\frac{f(z_k+th)-f(z_k)}{t}+\varepsilon\langle h_k, h\rangle\geq-2(\varepsilon+\alpha_k)\|h\|-\varepsilon=-5\varepsilon
	\end{eqnarray*}
	and consequently
	$$
	0\geq \inf_{\|h\|=1}g_{\varepsilon}'(z_k,h)\geq-5\varepsilon.
	$$
	Thanks to \text{\color{blue}Lemma 1} and \text{\color{blue}Proposition 3}, one can obtain that $\tau_{\min}(g_{\varepsilon},\bar x)\geq\frac{1}{5\varepsilon}$, which contradicts (iii) as $\varepsilon$ is arbitrary. The proof is complete.\end{proof} 

\begin{rema}(a) From \cite{MP2018,45}, the condition \eqref{3.9} means that $g$ is an $\varepsilon$-perturbation of $f$ near $\bar x$, and the condition $\inf_{\|h\|=1}f'(\bar x, h)\not=0$ is proved to be equivalent to the stability of this $\varepsilon$-perturbation of local error bounds. Further, it has been shown in \cref{theo:3.2}  that the stability of such $\varepsilon$-perturbation is essentially equivalent to that of $\varepsilon$-linear perturbation.
	
	(b) \cref{theo:3.2}  can be regarded as the equivalent version of \cite[Theorem 2]{45} since one can prove that
	$$
	\inf_{\|h\|=1}f'(\bar x, h)\not=0\Longleftrightarrow 0\not\in{\rm bdry}(\partial f(\bar x)).
	$$
	\begin{itemize}
		\item[]Indeed, suppose that $0\not\in{\rm bdry}(\partial f(\bar x))$. For the case that $0\in{\rm int}(\partial f(\bar x))$, there is $r>0$ such that $r\mathbf{B}^m\subseteq\partial f(\bar x)$. This and \eqref{2.4} imply that
		$$
		\inf_{\|h\|=1}f'(\bar x,h)\geq r>0.
		$$
		For the case that $0\not\in \partial f(\bar x)$, by the separation theorem, there exists $h_0\in\mathbb{R}^m$ with $\|h_0\|=1$ such that
		$$
		0>\sup\{\langle x^*, h_0\rangle: x^*\in \partial f(\bar x)\}=f'(\bar x,h_0)
		$$
		and consequently
		$$
		\inf_{\|h\|=1}f'(\bar x,h)\leq f'(\bar x,h_0)<0.
		$$
		
		On the other hand, if $0\in{\rm bdry}(\partial f(\bar x))$, then $\inf_{\|h\|=1}f'(\bar x,h)\geq 0$ and for any $\varepsilon>0$, we can select $u^*_{\varepsilon}\in \varepsilon\mathbf{B}^m\backslash\partial f(\bar x)$ and $x_{\varepsilon}\not=\bar x$ such that
		$$
		\langle u^*_{\varepsilon}, x_{\varepsilon}-\bar x\rangle>f(x_{\varepsilon})-f(\bar x).
		$$
		By \eqref{2.2}, for any $t\in(0,1)$, one has
		$$
		\frac{f(\bar x+t(x_{\varepsilon}-\bar x))-f(\bar x)}{t}\leq f(x_{\varepsilon})-f(\bar x)<\langle u^*_{\varepsilon}, x_{\varepsilon}-\bar x\rangle
		$$
		and thus
		$$
		f'\Big(\bar x,\frac{x_{\varepsilon}-\bar x}{\|x_{\varepsilon}-\bar x\|}\Big)\leq \Big\langle u^*_{\varepsilon},\frac{x_{\varepsilon}-\bar x}{\|x_{\varepsilon}-\bar x\|}\Big\rangle\leq\varepsilon.
		$$
		This means that $\inf_{\|h\|=1}f'(\bar x,h)\leq \varepsilon\rightarrow 0^+$ and so $\inf_{\|h\|=1}f'(\bar x,h)=0$.\hfill$\Box$
	\end{itemize}
\end{rema}

\subsection{Stability of Global Error Bounds}

This subsection is devoted to the study of stability of global error bounds for a single convex inequality, and the aim is to give sufficient and/or necessary conditions for the stability via directional derivatives. The following theorem gives a criterion for the stability of global error bounds.


\begin{theo}\label{theo:3.3}
	Let $f\in\Gamma_0(\mathbb{R}^m)$ be such that ${\rm bdry}(S_f)\subseteq f^{-1}(0)$. Consider the following statements:
	\begin{itemize}
		\item[\rm (i)] There exists $\tau\in (0, +\infty)$ such that
		\begin{equation}\label{3-9}
			\inf\left\{\Big|\inf_{\|h\|=1}f'(\bar x,h)\Big|: \bar x\in{\rm bdry}(S_f)\right\}>\tau.
		\end{equation}
		\item[\rm (ii)] There exist constants $c,\varepsilon\in (0, +\infty)$ such that for all $g\in\Gamma_0(\mathbb{R}^m)$ satisfying
		\begin{equation}\label{3-10}
			S_f\subseteq S_g\  \ {\it and} \ \ {\rm Lip}(f-g)<\varepsilon,
		\end{equation}
		one has $\tau_{\min}(g)\leq c$.
		\item[\rm (iii)] There exist constants $c,\varepsilon\in (0, +\infty)$ such that for all $g\in\Gamma_0(\mathbb{R}^m)$ satisfying
		\begin{equation}\label{3-11}
			{\rm bdry}(S_f)\cap g^{-1}(0)\not=\emptyset\  \ {\it and} \ \ {\rm Lip}(f-g)<\varepsilon,
		\end{equation}
		one has $\tau_{\min}(g)\leq c$.
	\end{itemize}
	Then {\rm (iii)}\,$\Rightarrow$\,{\rm (i)}$\Rightarrow$\,{\rm (ii)}.
\end{theo}

\begin{proof} (i)\,$\Rightarrow$\,(ii): If there is some $\bar x\in{\rm bdry}(S_f)$ such that $\inf_{\|h\|=1}f'(\bar x, h)>0$, then the implication follows by \text{\color{blue}Remark 4} and the proof of \cref{theo:3.2}.
	
	We next consider the case  $\inf_{\|h\|=1}f'(\bar x, h)\leq 0$ for all $\bar x\in{\rm bdry}(S_f)$.
	By virtue of \eqref{3-9} and \cref{theo:3.1}, one can verify that $S_f$ has a global error bound with the constant $\frac{1}{\tau}$; that is,
	\begin{equation}\label{3-12}
		d(x, S_f)\leq \frac{1}{\tau} [f(x)]_+,\ \ \forall x\in\mathbb{R}^m.
	\end{equation}
	Take any $\varepsilon\in (0,\tau)$. Suppose that $g\in\Gamma_0(\mathbb{R}^m)$ satisfies \eqref{3-10}. Let $x\in\mathbb{R}^m$ be such that $g(x)>0$. Then $f(x)>0$ as $S_f\subseteq S_g$. We claim that
	\begin{equation}\label{3.13}
		\inf_{\|h\|=1}f'(x, h)\leq -\tau.
	\end{equation}
	Granting this, by ${\rm Lip}(f-g)<\varepsilon$ in \eqref{3-10}, one can prove that
	$$
	\inf_{\|h\|=1}g'(x,h)\leq\inf_{\|h\|=1}f'(x,h)+\varepsilon\leq -(\tau-\varepsilon).
	$$
	This and \cref{theo:3.1} imply that $\tau_{\min}(g)\leq (\tau-\varepsilon)^{-1}$.
	
	We next prove the claim \eqref{3.13}. Take $z\in {\rm bdry}(S_f)$ such that $\|x-z\|=d(x,S_f)$ and \eqref{3-12} implies that
	$$
	f(x)\geq \tau d(x, S_f)=\tau \|x-z\|.
	$$
	Then for any $t\in (0,1)$, one has
	$$
	f(x+t(z-x))\leq tf(z)+(1-t)f(x)
	$$
	and thus
	$$
	\frac{f(x+t(z-x))-f(x)}{t}\leq -f(x)\leq-\tau\|x-z\|.
	$$
	This means that
	$$
	\inf_{\|h\|=1}f'(\bar x, h)\leq f'\Big(x,\frac{z-x}{\|x-z\|}\Big)\leq-\tau
	$$
	Hence \eqref{3.13} holds.
	
	
	(iii)\,$\Rightarrow$\,(i): Suppose that there exists a sequence $\{x_k\}\subseteq {\rm bdry}(S_f)$ such that
	$$
	\alpha_k:=\inf_{\|h\|=1}f'(x_k, h)\rightarrow 0^-\ ({\rm as}\ k\rightarrow \infty).
	$$
	Let $\varepsilon>0$ be arbitrary and $k$ be sufficiently large such that
	\begin{equation}\label{3.14}
		\frac{3}{2}\alpha_k+\frac{\varepsilon}{2}>0.
	\end{equation}
	Note that for any $x\not=x_k$, one has
	$$
	\frac{f(x)-f(x_k)}{\|x-x_k\|}=\frac{f\big(x_k+\|x-x_k\|\cdot\frac{x-x_k}{\|x-x_k\|}\big)-f(x_k)}{\|x-x_k\|}\geq f'\Big(x_k,\frac{x-x_k}{\|x-x_k\|}\Big)\geq\alpha_k
	$$
	and thus
	\begin{equation}\label{3.15}
		f(x)-\alpha_k\|x-x_k\|\geq f(x_k),\ \ \forall x\in\mathbb{R}^m.
	\end{equation}
	Choose $h_k\in\mathbb{R}^m$ with $\|h_k\|=1$ such that
	\begin{equation}\label{3.16}
		f'(x_k,h_k)<\alpha_k+\frac{\varepsilon}{2}.
	\end{equation}
	Then we can take $r_k\rightarrow 0^+$ (as $k\rightarrow \infty$) such that
	\begin{equation}\label{3.17}
		f(x_k+r_kh_k)<f(x_k)+(\alpha_k+\varepsilon)r_k.
	\end{equation}
	This and \eqref{3.15} imply that
	$$
	f(x_k+r_kh_k)-\alpha_k\|x_k+r_kh_k-x_k\|<\inf_{x\in\mathbb{R}^m}(f(x)-\alpha_k\|x-x_k\|)+\varepsilon r_k.
	$$
	Applying Ekeland variational principle, we can select $y_k\in\mathbb{R}^m$ such that
	\begin{equation}\label{3.18}
		\|y_k-(x_k+r_kh_k)\|<\frac{r_k}{2}, f(y_k)-\alpha_k\|y_k-x_k\|\leq  f(x_k+r_kh_k)-\alpha_k r_k,
	\end{equation}
	and
	\begin{equation}\label{3.19}
		f(x)-\alpha_k\|x-x_k\|+2\varepsilon\|x-y_k\|>f(y_k)-\alpha_k\|y_k-x_k\|,\ \ \forall x\not=y_k.
	\end{equation}
	This implies that
	$$
	\|y_k-x_k\|> r_k-\frac{r_k}{2}=\frac{r_k}{2}\ {\rm and}\ \|y_k-x_k\|< r_k+\frac{r_k}{2}=\frac{3}{2}r_k,
	$$
	and thus $y_k\not=x_k$. Let us consider a function $g_{\varepsilon}\in\Gamma_0(\mathbb{R}^m)$ defined by
	$$
	g_{\varepsilon}(x):=f(x)+\varepsilon\langle h_k, x-x_k\rangle\ \ {\rm for\ all} \ x\in\mathbb{R}^m.
	$$
	By virtue of \eqref{3.14}, \eqref{3.15}, \eqref{3.16} and \eqref{3.19}, one has
	\begin{eqnarray*}
		g_{\varepsilon}(y_k)=f(y_k)+\varepsilon\langle h_k, y_k-x_k\rangle&=&f(y_k)+\varepsilon\langle h_k, y_k-(x_k+r_kh_k)\rangle+\varepsilon r_k\\
		&\geq&\alpha_k\|y_k-x_k\|-\varepsilon\|y_k-(x_k+r_kh_k)\|+\varepsilon r_k\\
		&\geq&\alpha_k\cdot\frac{3}{2}r_k+\frac{\varepsilon}{2} r_k>0.
	\end{eqnarray*}
	If $\inf_{\|h\|=1}g_{\varepsilon}'(y_k, h)\geq 0$, then for any $x\not=y_k$, one has
	$$
	g_{\varepsilon}(x)-g_{\varepsilon}(y_k)\geq g_{\varepsilon}'\big(y_k,\frac{x-y_k}{\|x-y_k\|}\big)\|x-y_k\|\geq \inf_{\|h\|=1}g_{\varepsilon}'(y_k, h)\|x-y_k\|\geq 0.
	$$
	This and $g_{\varepsilon}(y_k)>0$ imply that $S_{g_{\varepsilon}}=\emptyset$, and thus $\tau_{\min}(g_{\varepsilon})=+\infty$, which contradicts (iii).
	
	Next, we consider the case $\inf_{\|h\|=1}g_{\varepsilon}'(y_k, h)<0$. For any $h\in\mathbb{R}^m$ with $\|h\|=1$ and $t>0$, by \eqref{3.19}, one has
	\begin{eqnarray*}
		\frac{g_{\varepsilon}(y_k+th)-g_{\varepsilon}(y_k)}{t}&=&\frac{f(y_k+th)-f(y_k)}{t}+\varepsilon\langle h_k, h\rangle\\
		&\geq&\frac{1}{t}\big(\alpha_k\|y_k+th-x_k\|-\alpha_k\|y_k-x_k\|-2\varepsilon \|y_k+th-y_k\|\big)+\varepsilon\langle h_k, h\rangle\\
		&\geq&\alpha_k-2\varepsilon-\varepsilon
	\end{eqnarray*}
	and consequently
	$$
	0>\inf_{\|h\|=1}g_{\varepsilon}'(y_k, h)\geq\alpha_k-2\varepsilon-\varepsilon\geq-4\varepsilon.
	$$
	Thanks to \text{\color{blue}Lemma 1} and \cref{theo:3.1}, we obtain $\tau_{\min}(g_{\varepsilon})\geq\frac{1}{4\varepsilon}$, which contradicts (iii) as $\varepsilon$ is arbitrary. The proof is complete.\end{proof}

\begin{rema} $(a)$ Compared with \cite[Theorem 7]{45} in which a subdifferential characterization of stability of global error bounds was established with the aid of the so-called asymptotic qualification condition, \cref{theo:3.3} studies the stability of global error bounds via directional derivatives without additional hypothesis.  It is known from \cref{theo:3.3} that the condition \eqref{3-9} is sufficient for the stability of global error bounds as said in (ii) of \cref{theo:3.3}, and is necessary for the stability as in (iii) of \cref{theo:3.3}.

	$(b)$ It should be  noted that the condition \eqref{3-9} is not sufficient for the stability of global error bounds as in (iii) of \cref{theo:3.3}, and the assumption $S_f\subseteq S_{g}$ for the stability as said in (ii) of \cref{theo:3.3} is crucial. To see this, let us consider the following example:
	\begin{exam}
		\item Let $f(x):=e^x-1$ for all $x\in\mathbb{R}$. Then $S_f=(-\infty, 0]$, ${\rm bdry}(S_f)=\{0\}$ and $|\inf_{|h|=1}f'(0, h)|=1>0$. However, for any $\varepsilon\in (0,+\infty)$, let us consider the function $g_{\varepsilon}(x):=f(x)-\varepsilon x$ for all $x\in\mathbb{R}$. Then one can verify that $g_{\varepsilon}$ has two different zero points which are denoted by $x_1:=\bar x<0$ and $x_2:=0$ and $S_{g_{\varepsilon}}=[\bar x, 0]$. Thus $S_f\not\subseteq S_{g_{\varepsilon}}$ and for any $x<\bar x$, one has
		$$
		\frac{d(x, S_{g_{\varepsilon}})}{g_{\varepsilon}(x)}=\frac{\bar x-x}{e^x-1-\varepsilon x}\rightarrow \frac{1}{\varepsilon} \ {\rm as} \ x\rightarrow -\infty.
		$$
		This implies that
		$$
		\tau_{\min}(g_{\varepsilon})\geq\frac{1}{2\varepsilon},
		$$
		and consequently the global stability (for $f$) as said in (iii) of \cref{theo:3.3}   does not hold as $\varepsilon>0$ is arbitrary.
	\end{exam}
	Further, a natural question arises  from the above example:
	\begin{center}
		{\it Does there exist some type of stability of global error bounds that can be characterized by condition \eqref{3-9}?}
	\end{center}
	We  do not have an answer to this question. However, if the answer is affirmative, we conjecture that such global stability should be strictly stronger than that of (ii) and weaker than that of (iii) in \cref{theo:3.3}.
	
\end{rema}


The following theorem gives characterizations of the stability of global error bounds for a convex inequality as said in (iii)  of \cref{theo:3.3}.

\begin{theo}\label{theo:3.4}
	Let $f\in\Gamma_0(\mathbb{R}^m)$ be such that ${\rm bdry}(S_f)\subseteq f^{-1}(0)$. Then the following statements are equivalent:
	\begin{itemize}
		\item[\rm (i)] There exists $\tau\in (0, +\infty)$ such that \eqref{3-9} holds and the following qualification condition is satisfied:\\
		$\mathbf{(QC)}$ For any sequence $\{z_k\}\subseteq S_f\backslash {\rm bdry}(S_f)$, one has
		\begin{equation}\label{3.20}
			\liminf_{k\rightarrow\infty}\Big|\inf_{\|h\|=1}f'(z_k,h)\Big|>\tau
		\end{equation}
		if there is a sequence $\{x_k\}\subseteq{\rm bdry}(S_f)$ satisfying $\lim_{k\rightarrow\infty}\frac{f(z_k)-f(x_k)}{\|z_k-x_k\|}=0$.
		\item[\rm (ii)] There exist constants $c,\varepsilon\in (0, +\infty)$ such that for all $g\in\Gamma_0(\mathbb{R}^m)$ satisfying \eqref{3-11}, one has $\tau_{\min}(g)\leq c$;
		\item[\rm (iii)] There exist constants $c,\varepsilon>0$ such that for any $\bar x\in{\rm bdry}(S_f)$ and $u\in\mathbb{R}^m$ with $\|u\|\leq 1$, one has $\tau_{\min}(g_{u, \varepsilon})\leq c$, where $g_{u, \varepsilon}(x):=f(x)+\varepsilon\langle u, x-\bar x\rangle$ for all $x\in \mathbb{R}^m$.
	\end{itemize}
\end{theo}
\begin{proof} (i)\,$\Rightarrow$\,(ii): Based on \text{\color{blue}Remark 4} and the proof of \cref{theo:3.3}, we only need to consider the case $\inf\limits_{\|h\|=1}f'(\bar x, h)\leq 0$ for all $\bar x\in{\rm bdry}(S_f)$.
	We first prove the following claim:\\
	
	\noindent{\it Claim: There exists $\varepsilon_0>0$ such that for all $x_0\in{\rm bdry}(S_f)$, one has}
	\begin{equation}\label{3.21}
		\inf\left\{\Big|\inf_{\|h\|=1}f'(z_0,h)\Big|: z_0\in\mathbb{R}^m, f(z_0)\geq-\varepsilon_0\|z_0-x_0\|\right\}\geq\tau.
	\end{equation}
	
	Suppose on the contrary that there exist $\varepsilon_k\rightarrow 0^+$, $x_k\in{\rm bdry}(S_f)$ and $z_k\in\mathbb{R}^m$ such that
	\begin{equation}\label{3.22}
		f(z_k)\geq-\varepsilon_k\|z_k-x_k\|\ \ {\rm and} \ \ \Big|\inf_{\|h\|=1}f'(z_k,h)\Big|<\tau\ \ {\rm for\ all } \ k.
	\end{equation}
	Then $f(z_k)\leq 0$ for all $k$ (otherwise, similar to the proof of \eqref{3.13}, one can prove that \\$\big|\inf_{\|h\|=1}f'(z_k,h)\big|>\tau$, a contradiction). By  \eqref{3-9}, one has $z_k\in S_f\backslash {\rm bdry}(S_f)$ and it follows from \eqref{3.22} that
	$$
	0\geq\frac{f(z_k)-f(x_k)}{\|z_k-x_k\|}=\frac{f(z_k)}{\|z_k-x_k\|}\geq-\varepsilon_k.
	$$
	This and the qualification condition in (i) imply  that
	$$\liminf_{k\rightarrow\infty}\Big|\inf_{\|h\|=1}f'(z_k,h)\Big|>\tau,
	$$ which contradicts \eqref{3.22}. Hence the claim is proved.\qed
	
	Let $\varepsilon>0$ be such that $\varepsilon<\min\{\varepsilon_0, \tau\}$. Suppose that $g\in\Gamma_0(\mathbb{R}^m)$ satisfies \eqref{3-11}. Take any $\bar x\in{\rm bdry}(S_f)\cap g^{-1}(0)$. Then for any $x\in\mathbb{R}^m$ with $g(x)>0$, one has
	$$
	f(x)\geq g(x)+(f(\bar x)-g(\bar x))-\varepsilon\|x-\bar x\|>-\varepsilon\|x-\bar x\|.
	$$
	Using   \eqref{3.21}, one obtains
	$$
	\inf_{\|h\|=1}f'(x,h)<-\tau
	$$
	and thus
	$$
	\inf_{\|h\|=1}g'(x,h)<\inf_{\|h\|=1}f'(x,h)+\varepsilon<-(\tau-\varepsilon).
	$$
	By  virtue of \text{\color{blue}Lemma 1} and \cref{theo:3.1} we derive the inequality $\tau_{\min}(g)\leq\frac{1}{\tau-\varepsilon}$.
	\vskip 2mm
	Note that (ii)\,$\Rightarrow$\,(iii) follows immediately and it remains to prove (iii)\,$\Rightarrow$\,(i).
	\vskip 2mm
	Suppose on the contrary that (i) does not hold. Based on (iii)\,$\Rightarrow$\,(i) in \cref{theo:3.3}, we only consider the case that there exist $z_k\in S_f\backslash{\rm bdry}(S_f)$ and $x_k\in {\rm bdry}(S_f)$ such that
	\begin{equation}\label{3.23}
		\lim_{k\rightarrow\infty}\frac{f(z_k)-f(x_k)}{\|z_k-x_k\|}=0\ \ {\rm and} \ \ \alpha_k:=\inf_{\|h\|=1}f'(z_k,h)\rightarrow 0^-.
	\end{equation}
	Let $\varepsilon>0$ be arbitrary. Without loss of generality, we can assume that $\frac{z_k-x_k}{\|z_k-x_k\|}\rightarrow h_0$ (considering subsequence if necessary). Then $\|h_0\|=1$. Suppose that $k$ is sufficiently large such that
	\begin{equation}\label{3.24}
		\alpha_k+\varepsilon>0\ \ {\rm and} \ \ \frac{f(z_k)-f(x_k)}{\|z_k-x_k\|}+\varepsilon\Big\langle h_0, \frac{z_k-x_k}{\|z_k-x_k\|}\Big\rangle>0.
	\end{equation}
	Let us consider a function $g_{h_0,\varepsilon}\in\Gamma_0(\mathbb{R}^m)$ defined by
	$$
	g_{h_0,\varepsilon}(x):=f(x)+\varepsilon\langle h_0, x-x_k\rangle\ \ {\rm for\ all} \ x\in\mathbb{R}^m.
	$$
	Then $g_{h_0,\varepsilon}(z_k)=f(z_k)+\varepsilon\langle h_0, z_k-x_k\rangle>0$ by \eqref{3.24} and thus
	$$
	0>\inf_{\|h\|=1}g_{h_0,\varepsilon}'(z_k,h)\geq\inf_{\|h\|=1}f'(z_k,h)-\varepsilon
	=\alpha_k-\varepsilon>-2\varepsilon.
	$$
	This together with \text{\color{blue}Lemma 1} and \cref{theo:3.1} implies  that $\tau_{\min}(g_{h_0,\varepsilon})\geq\frac{1}{2\varepsilon}$, which contradicts (iii) as $\varepsilon$ is arbitrary. The proof is complete. \end{proof}

\begin{rema} Note  that  condition $\mathbf{(QC)}$ in \eqref{3.20} is necessary for the stability of global error bounds. Consider \text{\color{blue}Example 9} given in \text{\color{blue}Remark 8} again. Let $f(x):=e^x-1$ for all $x\in\mathbb{R}$. Then the stability of global error bounds for $f$ as said in (iii) of \cref{theo:3.3} does not hold. Further, for any $z_k\rightarrow -\infty$, one can verify that
	$$
	\Big|\inf_{|h|=1}f'(z_k,h)\Big|=e^{z_k}\rightarrow 0 \ {\rm as} \ k\rightarrow \infty,
	$$
	which means that $\mathbf{(QC)}$ (for $f$) in \eqref{3.20} fails.\end{rema}

\section{Stability of error bounds for Semi-infinite Convex Constraint Systems}
In this section, we study local and global error bounds for  semi-infinite convex constraint systems, and mainly provide characterizations of stability of error bounds by directional derivatives. We first recall the definition of error bounds for semi-infinite convex constraint systems.

For semi-infinite convex constraint systems in $\mathbb{R}^m$, we mean  the problem of finding $x\in\mathbb{R}^m$ satisfying:
\begin{equation}\label{4.1}
	f_i(x)\leq 0\ \ {\rm for\ all }\ i\in I,
\end{equation}
where $I$ is a compact, possibly infinite, Hausdorff space, $f_i:\mathbb{R}^m\rightarrow \mathbb{R}, i\in I$, are given convex functions such that $i\mapsto f_i(x)$ is continuous on $I$ for each $x\in\mathbb{R}^m$. It is known from \cite[Theorem 7.10]{roc} that in this case, $(i,x)\mapsto f_i(x)$ is continuous on $I\times \mathbb{R}^m$.

Let $F\in C(I\times \mathbb{R}^m, \mathbb{R})$ be defined by $F(i,x):=f_i(x)$ for all $(i,x)\in I\times \mathbb{R}^m$. We denote the solution set of system \eqref{4.1} by
\begin{equation}\label{4.2}
	S_F:=\{x\in \mathbb{R}^m: f_i(x)\leq 0\  {\rm for\ all }\ i\in I\}.
\end{equation}
For any  $\ x\in\mathbb{R}^m$, we set
\begin{equation}\label{4.3}
	f(x):=\max\{f_i(x): i\in I\} \ \ {\rm and} \ \ I_f(x):=\{i\in I: f_i(x)=f(x)\}.
\end{equation}
Recall that system \eqref{4.1} is said to have a global error bound if there exists a constant
$\tau\in(0,+\infty)$ such that
\begin{equation}\label{4.4}
	d(x, S_F)\leq\tau [f(x)]_+\ \ \forall x\in\mathbb{R}^m.
\end{equation}
We denote by $\tau_{\min}(F):=\inf\{\tau>0:\eqref{4.4}\ {\rm holds}\}$ the global error bound modulus of $S_F$.

For $\bar x\in{\rm bdry}(S_F)$, system \eqref{4.1} is said to have a local error bound at $\bar x$ if there exist constants
$\tau,\delta\in(0,+\infty)$ such that
\begin{equation}\label{4.5}
	d(x, S_F)\leq\tau [f(x)]_+\ \ \forall x\in B(\bar x,\delta).
\end{equation}
We denote by $\tau_{\min}(F,\bar x):=\inf\{\tau>0:\ {\rm there\ exists} \ \delta>0 \ {\rm such \ that}\ \eqref{4.5}\ {\rm holds}\}$ the local error bound modulus of $S_F$ at $\bar x$.

We first study stability of local error bounds for semi-infinite convex constraint system \eqref{4.1} and aim to provide characterizations of the stability of local error bounds for system \eqref{4.1}. To this aim, we need the following proposition which is of independent interest.

\begin{prop}\label{prop:4.1}
	Let $x\in\mathbb{R}^m$. Then for any $h\in \mathbb{R}^m$, one has
	\begin{equation} \label{4.6}
		f'(x, h)=\max_{i\in I_f(x)}f_i'(x, h).
	\end{equation}
\end{prop}

\begin{proof}  Let $h\in \mathbb{R}^m$. Take any $i\in I_f(x)$. Then for any $t>0$, one has
	$$
	\frac{f_i(x+th)-f_i(x)}{t}\leq \frac{f(x+th)-f(x)}{t}
	$$
	and thus $f_i'(x,h)\leq f'(x,h)$. This implies that
	\begin{equation}\label{4.7}
		f'(x, h)\geq\max_{i\in I_f(x)}f_i'(x, h).
	\end{equation}
	By virtue of \eqref{2.4}, one has
	$$
	f'(x,h)=\max_{x\in\partial f(x)}\langle x^*, h\rangle,
	$$
	and thus there is $z^*\in\partial f(x)$ such that
	\begin{equation}\label{4.8}
		f'(x,h)=\langle z^*, h\rangle.
	\end{equation}
	Note that the subdifferential of the function $f$ at a point $x\in\mathbb{R}^m$ is given by (see Ioffe \& Tikhomirov \cite{SIAM13})
	$$
	\partial f(x)={\rm co}\Big(\bigcup_{i\in I_f(x)}\partial f_i(x)\Big)
	$$
	where ``${\rm co}$" denotes the convex hull of a set. Then by \eqref{4.8}, there exist $\lambda_1,\cdots,\lambda_N\geq 0$, $i_1,\cdots,i_N\in I_f(x)$ and $z_k^*\in\partial f_{i_k}(x), k=1,\cdots, N$ such that
	\begin{equation*}
		\sum_{k=1}^N\lambda_{k}=1\ \ {\rm and} \ \ z^*=\sum_{k=1}^N\lambda_{k}z_k^*.
	\end{equation*}
	This and \eqref{4.8} imply that
	$$
	f'(x,h)=\langle z^*, h\rangle=\sum_{k=1}^N\lambda_{k}\langle z_k^*, h\rangle\leq\sum_{k=1}^N\lambda_{k}f'_{i_k}(x,h)\leq\max_{i\in I_f(x)}f'_i(x,h).
	$$
	Hence \eqref{4.6} follows from \eqref{4.7} and the above inequality. The proof is complete.\end{proof} 

The following theorem gives characterizations (by directional derivatives) of stability of local error bounds for system \eqref{4.1}.

\begin{theo}\label{theo:4.1}
	Let $\bar x\in \mathbb{R}^m$ be such that $f(\bar x)=0$. Then the following statements are equivalent:
	\begin{itemize}
		\item[\rm(i)] $\inf_{\|h\|=1}f'(\bar x, h)\not=0$.
		\item[\rm(ii)] There exist constants $c,\varepsilon>0$ such that if
		\begin{eqnarray*}
			&G\in C(I\times\mathbb{R}^m, \mathbb{R}), g_i(x):=G(i,x), g_i \ {\it is \ convex};\\
			&g(x):=\max_{i\in I}g_i(x),  I_g(x):=\{i\in I: g_i(x)=g(x)\};\\
			&g(\bar x)=0;\\
			&I_g(\bar x)\subseteq I_f(\bar x)\ {\it whenever}\ \inf_{\|h\|=1}f'(\bar x, h)<0;\\
			&I_f(\bar x)\subseteq I_g(\bar x)\ {\it whenever}\ \inf_{\|h\|=1}f'(\bar x, h)>0;\\
			&\limsup\limits_{x\rightarrow\bar x}\frac{|f_i(x)-g_i(x)-(f_i(\bar x)-g_i(\bar x))|}{\|x-\bar x\|}\leq\varepsilon, \ \forall i\in I_f(\bar x)\cap I_g(\bar x),
		\end{eqnarray*}
		then one has $\tau_{\min}(G,\bar x)\leq c$.
		\item[\rm(iii)] There exist constants $c,\varepsilon>0$ such that for all $u^*\in\mathbb{R}^m$ with $\|u^*\|\leq 1$, one has $\tau_{\min}(G,\bar x)\leq c$, where $G\in C(I\times\mathbb{R}^m, \mathbb{R})$ is defined by
		\begin{equation}\label{4.15}
			G(i,x):=f_i(x)+\varepsilon\langle u^*, x-\bar x\rangle\ \ {\it for\ all} \ (i,x)\in I\times\mathbb{R}^m.
		\end{equation}
	\end{itemize}
\end{theo}

\begin{proof} We set
	$$
	\beta(f,\bar x):=\inf_{\|h\|=1}f'(\bar x, h).
	$$
	(i)\,$\Rightarrow$\,(ii): Suppose that $\beta(f,\bar x)>0$. Then one can verify that $S_F=\{\bar x\}$ by \text{\color{blue}Remark 4}. Choose any $\varepsilon\in (0,\beta(f,\bar x))$. Suppose that $G,g_i$ and $g$ satisfy all conditions said in (ii). Then for any $i\in I_f(\bar x)\subseteq I_g(\bar x)$, one has
	$$
	g_i'(\bar x,h)\geq f_i(\bar x, h)-\varepsilon
	$$
	and it follows from \cref{prop:4.1} that
	\begin{eqnarray*}
		\inf_{\|h\|=1}g'(\bar x,h)=\inf_{\|h\|=1}\max_{i\in I_g(\bar x)}g_i'(\bar x,h)&\geq& \inf_{\|h\|=1}\max_{i\in I_f(\bar x)}g_i'(\bar x,h)\\
		&\geq& \Big(\inf_{\|h\|=1}\max_{i\in I_f(\bar x)}f_i'(\bar x,h)\Big)-\varepsilon\\
		&=&\inf_{\|h\|=1}f'(\bar x,h)-\varepsilon\\
		&=&\beta(f,\bar x)-\varepsilon>0
	\end{eqnarray*}
	(thanks to $I_f(\bar x)\subseteq I_g(\bar x)$). Applying \text{\color{blue}Proposition 3}, we derive the inequality
	$$
	\tau_{\min}(G,\bar x)=\tau_{\min}(g,\bar x)\leq\frac{1}{\beta(f,\bar x)-\varepsilon}.
	$$
	Suppose that $\beta(f,\bar x)<0$. Choose any $\varepsilon>0$ such that $\beta(f,\bar x)+\varepsilon<0$. Then for any $i\in I_g(\bar x)\subseteq I_f(\bar x)$, one has
	$$
	g_i'(\bar x,h)\leq f_i(\bar x, h)+\varepsilon
	$$
	and it follows from \cref{prop:4.1} that
	\begin{eqnarray*}
		\inf_{\|h\|=1}g'(\bar x,h)=\inf_{\|h\|=1}\max_{i\in I_g(\bar x)}g_i'(\bar x,h)&\leq& \inf_{\|h\|=1}\max_{i\in I_g(\bar x)}(f_i'(\bar x,h)+\varepsilon)\\
		&\leq& \Big(\inf_{\|h\|=1}\max_{i\in I_f(\bar x)}f_i'(\bar x,h)\Big)+\varepsilon\\
		&=&\inf_{\|h\|=1}f'(\bar x,h)+\varepsilon\\
		&=&\beta(f,\bar x)+\varepsilon<0
	\end{eqnarray*}
	(thanks to $I_g(\bar x)\subseteq I_f(\bar x)$). Applying \text{\color{blue}Proposition 3} again, we obtain the inequality
	$$
	\tau_{\min}(G,\bar x)=\tau_{\min}(g,\bar x)\leq\frac{1}{-\beta(f,\bar x)-\varepsilon}.
	$$

	(ii)\,$\Rightarrow$\,(iii): The implication follows immediately as $I_f(\bar x)=I_g(\bar x)$.
	
	(iii)\,$\Rightarrow$\,(i): Let $u^*\in\mathbb{R}^m$ with $\|u^*\|\leq 1$ and $G\in C(I\times \mathbb{R}^m, \mathbb{R})$ be defined as \eqref{4.15}. Note that
	$$
	g_i(x)=G(i,x)=f_i(x)+\varepsilon\langle u^*, x-\bar x\rangle
	$$
	and thus
	$$
	g(x)=\max_{i\in I}g_i(x)=\max_{i\in I}(f_i(x)+\varepsilon\langle u^*, x-\bar x\rangle)=f(x)+\varepsilon\langle u^*, x-\bar x\rangle.
	$$
	This means that the implication follows from (iii)\,$\Rightarrow$\,(i) as in
	\cref{theo:3.2}.
	The proof is complete.\end{proof}

\begin{rema}
	(a) \cref{theo:4.1}, given in terms of directional derivatives, can be regarded as an equivalent version and a supplement of \cite[Theorem 4]{45} in which a subdifferential characterization of stability of local error bounds for system \eqref{4.1} was established. Further, in contrast with \cite[Theorem 4]{45}, the stability of local error bounds for system \eqref{4.1} only requires that all component functions in system \eqref{4.1} have the same $\varepsilon$-linear perturbation.
	
	(b) It  should be observed that the condition $I_f(\bar x)\subseteq I_g(\bar x)$ or $I_g(\bar x)\subseteq I_f(\bar x)$ in \cref{theo:4.1} is crucial. To see this, we consider the following two examples:
	\begin{itemize}
		\item[$\triangle$]\textit{Let $f_i:\mathbb{R}^2\rightarrow \mathbb{R}$ be defined by $f_i(x):=|x_i|,i=1,2$ for all $x=(x_1,x_2)\in\mathbb{R}^2$, $\bar x=(0,0)$, $F:=(f_1,f_2)$ and $f:=\max\{f_1,f_2\}$. Then
			$$I_f(\bar x)=\{1,2\}\  \ {\it and} \ \ \inf_{\|h\|=1}f'(\bar x,h)=\frac{\sqrt{2}}{2}>0.
			$$
			However, for each $\varepsilon>0$, we define functions $g_{1,\varepsilon}$ and $g_{2,\varepsilon}$ by
			$$
			g_{1,\varepsilon}(x):=|x_1|+\varepsilon|x_2|, g_{2,\varepsilon}(x):=|x_2|-\varepsilon, \ \ {\it for\ all} \ x=(x_1,x_2)\in\mathbb{R}^2.
			$$
			We set $G_{\varepsilon}:=(g_{1,\varepsilon}, g_{2,\varepsilon})$ and $g_{\varepsilon}:=\max\{g_{1,\varepsilon}, g_{2,\varepsilon}\}$. Then one can verify that $I_g(\bar x)=\{1\}$ and thus $I_f(\bar x)\not\subseteq I_g(\bar x)$. Note that
			$$
			S_{G_{\varepsilon}}=\{\bar x\}\ \ {\it and} \ \ {\rm Lip}(f_1-g_{1,\varepsilon})\leq \varepsilon.
			$$
			For any $\delta\in  (0,\varepsilon^{-1})$, we set $z_{\delta}:=(0,\delta)\in\mathbb{R}^2$. Then $d(z_{\delta}, S_{G_{\varepsilon}})=\delta$ and $g_{\varepsilon}(z_{\delta})=\varepsilon\delta$, which implies that $\tau_{\min}(G_{\varepsilon},\bar x)\geq\frac{1}{\varepsilon}$.}
		\item[$\triangle\triangle$]\textit{Let $f_1,f_2:\mathbb{R}^2\rightarrow \mathbb{R}$ be defined by $f_1(x):=x_1$ and $f_1(x):=-x_1+|x_2|-1$ for all $x=(x_1,x_2)\in\mathbb{R}^2$, $\bar x=(0,0)$, $F:=(f_1,f_2)$ and $f:=\max\{f_1,f_2\}$. Then
			$$I_f(\bar x)=\{1\}\  \ {\it and} \ \ \inf_{\|h\|=1}f'(\bar x,h)=\inf_{\|h\|=1}f_1'(\bar x,h)=-1<0.
			$$
			However, for each $\varepsilon>0$, we define functions $g_{1,\varepsilon}$ and $g_{2,\varepsilon}$ as
			$$
			g_{1,\varepsilon}(x):=x_1+\varepsilon|x_2|, g_{2,\varepsilon}(x):=-x_1+\varepsilon|x_2|, \ \ {\rm for\ all} \ x=(x_1,x_2)\in\mathbb{R}^2.
			$$
			We set $G_{\varepsilon}:=(g_{1,\varepsilon}, g_{2,\varepsilon})$ and $g_{\varepsilon}:=\max\{g_{1,\varepsilon}, g_{2,\varepsilon}\}$. Then one can verify that $I_g(\bar x)=\{1,2\}$ and thus $I_g(\bar x)\not\subseteq I_f(\bar x)$. Note that
			$$
			S_{G_{\varepsilon}}=\{\bar x\}\ \ {\it and} \ \ {\rm Lip}(f_1-g_{1,\varepsilon})\leq \varepsilon.
			$$
			For any $\delta\in  (0,\varepsilon^{-1})$, set $z_{\delta}:=(0,\delta)\in\mathbb{R}^2$. Then $d(z_{\delta}, S_{G_{\varepsilon}})=\delta$ and $g_{\varepsilon}(z_{\delta})=\varepsilon\delta$. This means that $\tau_{\min}(G_{\varepsilon},\bar x)\geq\varepsilon^{-1}$.}
	\end{itemize}
\end{rema}

We now turn our attention to the stability of global error bounds for semi-infinite constraint system \eqref{4.1} and mainly give equivalent criterion for such stability. Based on \cref{theo:3.4}, the following theorem establishes equivalent conditions for the stability of global error bounds for the  system \eqref{4.1}.
\begin{theo}\label{theo:4.2}
	The following statements are equivalent:
	\begin{itemize}
		\item[\rm(i)] There exists $\tau\in (0, +\infty)$ such that
		\begin{equation}\label{4.16}
			\inf\left\{\Big|\inf_{\|h\|=1}f'(\bar x,h)\Big|: \bar x\in{\rm bdry}(S_f)\right\}>\tau,
		\end{equation}
		and $\mathbf{(QC)}$ as in \cref{theo:3.4} is satisfied.
		
		\item[\rm(ii)] There exist constants $c,\varepsilon>0$ such that if
		\begin{eqnarray*}
			&G, g_i(x), g(x)\ {\it and} \  I_g(x) \ {\it as\ said \ in\ (ii)\ of\ Theorem \ 13};\\                                                                      &\big\{z\in{\rm bdry}(S_f): f_i(z)=g_i(z) \  {\it for\ all} \ i\in I\big\}\not=\emptyset;  \\
			&\sup_{i\in I}{\rm Lip}(f_i-g_i)<\varepsilon;\\
			&I_g(x)\subseteq I_f(x)\ {\it whenever}\ \inf_{\|h\|=1}f'(x, h)< 0;\\
			&I_f(x)\subseteq I_g(x)\ {\it whenever}\ \inf_{\|h\|=1}f'(x, h)>0,
		\end{eqnarray*}
		then one has $\tau_{\min}(G)\leq c$.
		\item[\rm(iii)] There exist constants $c,\varepsilon>0$ such that for all $\bar x\in{\rm bdry}(S_f)$ and $u^*\in\mathbb{R}^m$ with $\|u^*\|\leq 1$, one has $\tau_{\min}(G)\leq c$, where $G\in C(I\times\mathbb{R}^m, \mathbb{R})$ is defined by
		\begin{equation}\label{4.23}
			G(i,x):=f_i(x)+\varepsilon\langle u^*, x-\bar x\rangle\ \ {\it for\ all} \ (i,x)\in I\times\mathbb{R}^m.
		\end{equation}
	\end{itemize}
\end{theo}

\begin{proof} (i)\,$\Rightarrow$\,(ii): Thanks to \text{\color{blue}Remark 4} and the proof of \text{\color{blue} Theorem 13}, we only need to consider the case $\inf_{\|h\|=1}f'(\bar x, h)\leq 0$ for all $\bar x\in{\rm bdry}(S_f)$. By virtue of the claim given  in the proof of \cref{theo:3.4}, there exists $\varepsilon_0>0$ such that for all $x_0\in{\rm bdry}(S_f)$, one has
	\begin{equation}\label{4.24}
		\inf\left\{\Big|\inf_{\|h\|=1}f'(x,h)\Big|: x\in\mathbb{R}^m, f(x)\geq-\varepsilon_0\|x-x_0\|\right\}\geq\tau.
	\end{equation}
	Take any $\varepsilon>0$ such that $\varepsilon<\min\{\varepsilon_0,\tau\}$. Suppose that $G,g_i$ and $g$ satisfy all conditions said in (ii). Let $x\in\mathbb{R}^m$ be such that $g(x)>0$. We claim that
	\begin{equation}\label{4.25}
		\inf_{\|h\|=1}f'(x,h)\leq-\tau.
	\end{equation}
	Granting this, one has
	\begin{eqnarray*}
		\inf_{\|h\|=1}g'(x,h)=\inf_{\|h\|=1}\Big(\max_{i\in I_g(x)}g_i'(x,h)\Big)&\leq&\inf_{\|h\|=1}\Big(\max_{i\in I_g(x)}f_i'(x,h)+\varepsilon\Big)\\
		&\leq&\inf_{\|h\|=1}\Big(\max_{i\in I_f(x)}f_i'(x,h)\Big)+\varepsilon\\
		&=&\inf_{\|h\|=1}f'(x,h)+\varepsilon\\
		&\leq&-(\tau-\varepsilon)
	\end{eqnarray*}
	(thanks to $\sup_{i\in I}{\rm Lip}(f_i-g_i)<\varepsilon$ and $I_g(x)\subseteq I_f(x)$). Applying  \text{\color{blue}Lemma 1} and \cref{theo:3.1}, we derive the inequality $\tau_{\min}(G)\leq \frac{1}{\tau-\varepsilon}$.
	\vskip 2mm
	It remains to prove relation  \eqref{4.25}. For the case that $f(x)>0$, similar to the proof  of \eqref{3.13}, one can verify that \eqref{4.25} holds. Thus we only need to consider the case that $f(x)\leq 0$.
	
	From the second condition in (ii), there is $z_0\in{\rm bdry}(S_f)$ such that $f_i(z_0)=g_i(z_0)$ for all $i\in I$. Then for any $i\in I_g(x)\subseteq I_f(x)$, one has
	\begin{eqnarray*}
		f_i(x)\geq g_i(x)-(f_i(z_0)-g_i(z_0))-\varepsilon\|x-z_0\|=g(x)-\varepsilon\|x-z_0\|>-\varepsilon\|x-z_0\|
	\end{eqnarray*}
	and thus $f(x)>-\varepsilon_0\|x-z_0\|$. This and \eqref{4.24} imply that \eqref{4.25} holds.
	
	(ii)\,$\Rightarrow$\,(iii): The implication follows immediately since $I_f(x)=I_g(x)$ for all $x\in\mathbb{R}^m$.

	(iii)\,$\Rightarrow$\,(i): Let $\bar x\in{\rm bdry}(S_f)$, $u^*\in\mathbb{R}^m$ with $\|u^*\|\leq 1$ and $G\in C(I\times \mathbb{R}^m, \mathbb{R})$ be defined as \eqref{4.23}. Note that
	$$
	g_i(x)=G(i,x)=f_i(x)+\varepsilon\langle u^*, x-\bar x\rangle
	$$
	and consequently
	$$
	g(x)=\max_{i\in I}g_i(x)=\max_{i\in I}(f_i(x)+\varepsilon\langle u^*, x-\bar x\rangle)=f(x)+\varepsilon\langle u^*, x-\bar x\rangle.
	$$
	Thus,  the implication follows from (iii)\,$\Rightarrow$\,(i) as in \cref{theo:3.4}. The proof is complete
\end{proof}%

\section{Conclusions}

This paper is devoted to  the study of stability of  local and global error bounds for convex inequality constraint systems including a single convex inequality and  semi-infinite convex constraint systems. The main results provide characterizations (in terms of directional derivatives) of stability of local and global error bounds for a single convex inequality (see \cref{theo:3.2} and \cref{theo:3.4}). When these results are applied to error bounds for semi-infinite convex constraint systems, characterizations of the stability of local and global error bounds are also established in terms of directional derivatives (see \cref{theo:4.1} and  \cref{theo:4.2}). These results show that the stability of error bounds for convex inequality constraint systems can be equivalent to solving the convex optimization/minimization problems defined by directional derivatives over the unit sphere.\\

\noindent{\bf Acknowledgements }The authors are grateful to Professor Constantin Zalinescu for  his constructive remark concerning Lemma 1.

\end{document}